\documentclass{article}

\usepackage{fullpage,amsthm,amsmath,amssymb,amsbsy,amsfonts}

\usepackage{graphicx}

\title{\bf Invariant convex bodies for strongly elliptic systems}

\author{\sc{Gershon Kresin$^a\!\!$}
 and \sc{Vladimir Maz'ya$^b$}
\\ \\
{\it{$^a$Department of Computer Science and Mathematics, Ariel University,
Ariel 40700, Israel}}\\
{\it{$^b$Department of Mathematical Sciences, University of Liverpool,
M$\&$O Building, Liverpool,}}\\ 
{\it{L69 3BX, UK; Department of Mathematics, Link\"oping University,SE-58183 Link\"oping,}}\\ 
{\it{{\hskip -145mm}Sweden}}
}
{ \date\ }

\numberwithin{equation}{section}
\newtheorem{lemma}{Lemma}
\newtheorem{theorem}{Theorem}
\newtheorem{proposition}[theorem]{Proposition}

\newcommand{\bs}{\boldsymbol}

\newcommand{\blankbox}{\hfill\qedsymbol\\}

\begin{document}
\maketitle
\large

{\bf Abstract.} We consider uniformly strongly elliptic systems of the second order with bounded coefficients. 
First, sufficient conditions for the invariance of convex bodies obtained for linear 
systems without zero order term in bounded domains and quasilinear systems of special form
in bounded and in a class of unbounded domains. 
These conditions are formulated 
in algebraic form. They describe relation between the geometry of the 
invariant convex body and the coefficients of the system. 
Next, necessary conditions, which are also sufficient, for 
the invariance of some convex bodies are found for elliptic homogeneous 
systems with constant coefficients in a half-space. 
The necessary conditions are derived by using a criterion on the 
invariance of convex bodies for normalized matrix-valued integral transforms
also obtained in the paper. In contrast with the previous studies of invariant sets
for elliptic systems no a priori restrictions on the coefficient matrices are imposed.
\\
\\ 

\section{Main results and background}\label{CH_10I}

We consider linear systems of the form
\begin{equation} \label{S1}
{\mathfrak A}(x, D_x) \bs u=\sum ^n_{j,k=1}{\mathcal A}_{jk}(x){\partial ^2 \bs u
\over\partial x_j\partial x_k}+\sum ^n_{j=1}{\mathcal A}_{j}(x){\partial \bs u
\over\partial x_j}=\bs 0 
\end{equation}
and certain quasilinear systems of the second order. 
Here $D_x=(\partial /\partial x_1,\dots,\partial /\partial x_n)$,
$\bs u=(u_1,\dots,u_m)$, ${\mathcal A}_{jk}$ 
and ${\mathcal A}_{j}$  
are bounded real $(m\times m)$-matrix-valued functions in a proper subdomain 
$\Omega$ of the Euclidean space ${\mathbb R}^n$ with boundary $\partial \Omega$ 
and closure $\overline \Omega$. Without loss of generality we suppose that
${\mathcal A}_{jk}={\mathcal A}_{kj}$.
We assume that the operator ${\mathfrak A}(x, D_x)$ 
is uniformly strongly elliptic in $\Omega$, i.e. that the inequality
\begin{equation} \label{(0.2)}
\Biggl (\;\sum_{j,k=1}^{n} {\mathcal A} _{jk}(x)\sigma _j \sigma _k \bs\zeta , \bs\zeta \Biggr )
\geq \delta|\bs\sigma|^2|\bs\zeta|^2 
\end{equation}
holds with a positive constant $\delta$ for all vectors 
$\bs\sigma  =(\sigma _1,\dots, \sigma  _n)$, $\bs\zeta =(\zeta_1,\dots, \zeta _m )$ 
and points $x\in\Omega$. Here and henceforth by $|\cdot|$ and $(\cdot,\cdot)$ 
we denote the length of a vector and the inner product in the Euclidean space.

We are  interested in conditions for the invariance of sets for system (\ref{S1})
and some quasilinear systems.
We will not suppose beforehand that the principal 
part of a system under consideration 
satisfies structural restrictions such as scalarity or diagonality.

The notion of invariant set for parabolic and elliptic systems and the first 
results concerning these sets appeared in the paper by Weinberger \cite{WEIN}.
By definition, a set ${\mathcal S}\subset {\mathbb R}^m$ 
is called invariant for elliptic system of the second order in a domain
$\Omega$ if any continuous in $\overline \Omega$ and 
bounded classical solution $\bs u=(u_1,\dots,u_m)$ of this system belongs to 
${\mathcal S}$ under the assumption that $\bs u|_{\partial \Omega} \in {\mathcal S}$.
It is noted in \cite{WEIN} that the componentwise maximum principle and the 
classical maximum modulus principle for parabolic and elliptic systems can be 
interpreted as statements on the invariance of an orthant and a ball, respectively.

Henceforth by ${\mathfrak S}$ we denote the closure of an arbitrary convex 
proper subdomain of ${\mathbb R}^m$. For brevity we say that
${\mathfrak S}$ is a {\sl convex body}.
By $\partial^*{\mathfrak S}$ we mean the set of points $a \in 
\partial{\mathfrak S}$ for which there exists the unit 
outward normal $\bs \nu(a)$ to $\partial{\mathfrak S}$.
We use the notation ${\mathfrak N}_{\mathfrak S}=\{\bs \nu(a): 
a \in \partial^*{\mathfrak S} \}$. Here end in the sequel $^{t}\!{\mathcal A}$ stands for 
the transposed matrix of ${\mathcal A}$.

\smallskip
In section 2 we find the following sufficient condition for the 
invariance of convex bodies for system (\ref{S1}).  

\smallskip
\begin{theorem} \label{T_1} Let $\Omega$ be a bounded domain in ${\mathbb R}^n $.
Let ${\mathfrak S}$ be a convex body in ${\mathbb R}^m$ and let the coefficients of the system 
${\mathfrak A}(x, D_x)\bs u=\bs 0$ in $\Omega$ satisfy the equalities
\begin{equation} \label{(4.1A)}
^{t}\!\!{\mathcal A}_{jk}(x)\bs\nu =a _{jk}(x;\bs\nu)\bs\nu\;,\;\;\;\;\;\;
^{t}\!\!{\mathcal A}_{j}(x)\bs\nu =a _{j}(x;\bs\nu)\bs\nu
\end{equation}
for all $x\in \Omega$  and $\bs\nu\in{\mathfrak N}_{\mathfrak S}$ with 
$a_{jk}, a_j: \Omega\times {\mathfrak N}_{\mathfrak S}\to {\mathbb R}$, 
$1\leq j,k\leq n$.

Then ${\mathfrak S}$ is invariant for the system ${\mathfrak A}(x, D_x)\bs u=\bs 0$ in $\Omega$.
\end{theorem}

\smallskip
Quasilinear systems of the form
\begin{equation} \label{Q1}
{\mathfrak B}(x, D_x) \bs u=\sum ^n_{j,k=1}{\mathcal B}_{jk}(x, D_x \bs u)
{\partial ^2 \bs u \over\partial x_j\partial x_k}=\bs 0 
\end{equation}
in bounded and in a wide class of unbounded domains $\Omega $
are also considered in Section 2,
where $\bs u=(u_1,\dots, u_m)$, ${\mathcal B}_{jk}$ 
are bounded real $(m\times m)$-matrix-valued functions in $\Omega \times {\mathbb R}^{mn}$. 
Without loss of generality we suppose that
${\mathcal B}_{jk}={\mathcal B}_{kj}$.
We assume that the operator ${\mathfrak B}(x, D_x)$ 
is uniformly strongly elliptic in $\Omega$, i.e. that the inequality
\begin{equation} \label{(Q2)}
\Biggl (\;\sum_{j,k=1}^{n} {\mathcal B} _{jk}(x, \bs\eta )
\sigma _j \sigma _k \bs\zeta , \bs\zeta \Biggr )
\geq \delta|\bs\sigma|^2|\bs\zeta|^2 
\end{equation}
holds with a positive constant $\delta$ for all vectors 
$\bs\eta \in {\mathbb R}^{mn}, \bs\sigma \in {\mathbb R}^n$, $\bs\zeta \in {\mathbb R}^m$ 
and points $x\in\Omega$.

In the next assertion we describe a sufficient condition for the invariance of convex bodies 
for system (\ref{Q1}).

\begin{theorem} \label{T_2} Let $\Omega \subset {\mathbb R}^n$ be 
{\rm (i)} a bounded domain or 
{\rm (ii)} an unbounded domain such that the cone 
$$
K_h=\left \{ x\in {\mathbb R}^n:\; x_n^2> h^2 \sum_{i=1}^{n-1}x^2_i\;,\;x_n<0 \right \},
\;\;\;h>1,
$$ 
belongs to the complement of $\Omega$.

Let ${\mathfrak S}$ be a convex body in ${\mathbb R}^m$ and let the coefficients of the system 
${\mathfrak B}(x, D_x)\bs u=\bs 0$ in $\Omega$ satisfy the equalities
\begin{equation} \label{(Q3)}
^{t}{\mathcal B}_{jk}(x, \bs \eta)\bs\nu =b _{jk}(x, \bs \eta;\bs\nu)\bs\nu
\end{equation}
for all $x\in \Omega$, $\bs \eta \in {\mathbb R}^{mn}$ and 
$\bs\nu\in{\mathfrak N}_{\mathfrak S}$ with 
$b_{jk}: \Omega\times {\mathbb R}^{mn} \times {\mathfrak N}_{\mathfrak S}\to {\mathbb R}$, 
$1\leq j,k\leq n$.

Then ${\mathfrak S}$ is invariant for the system ${\mathfrak B}(x, D_x)\bs u=\bs 0$ 
in $\Omega$.
\end{theorem}

\smallskip
In section 3 we explore the structure of an $(m\times m)$-matrix ${\mathcal A}$ 
satisfying condition 
\begin{equation} \label{(Q4)}
^t\!{\mathcal A}\bs\nu=a(\bs\nu)\bs\nu
\end{equation}
for any $\bs\nu \in {\mathfrak N}_{\mathfrak S}$,  where ${\mathfrak S}$ is a convex polyhedral angle, 
a cylindrical or conical body
and $a$ is a scalar function on ${\mathfrak N}_{\mathfrak S}$. 

For instance, we show that the matrix ${\mathcal A}$ is scalar if ${\mathfrak S}$ is a convex polyhedral 
cone with $p$ facets, $p>m$, a convex cone with smooth guide and convex compact body with smooth boundary.

In other case, it is shown that if ${\mathfrak S}$ is a convex polyhedral cone with $m$ facets then 
the matrix ${\mathcal A}$ is represented in the form
$$
{\mathcal A}=\big (\;^{t}[\bs\nu_1,\dots, \bs\nu_m] \big)^{-1}\;
{\mathcal D}\;^{t}[\bs\nu_1,\dots, \bs\nu_m],
$$
where ${\mathcal D}$ is a diagonal $(m\times m)$-matrix, $\bs\nu_k$ is the unit outward normal to 
$k$-th facet of the polyhedral cone and $[\bs\nu_1,\dots, \bs\nu_m]$ means the $(m\times m)$-matrix 
whose columns are $\bs\nu_1,\dots, \bs\nu_m$. In particular, if ${\mathfrak S}$ is the first orthant
${\bf R}^m_+=\{u=(u_1,\dots,u_m): u_1\geq 0,\dots,u_m \geq 0 \}$ then ${\mathcal A}$ is diagonal.

At the end of section 3 we give examples of matrices ${\mathcal A}$ satisfying condition (\ref{(Q4)})
for certain three-dimensional convex bodies.

\smallskip
The results of auxiliary section 4 are used in  section 5.
In section 4 we consider the matrix-valued integral transform 
\begin{equation} \label{EQ_101}
(T\bs u)(x)=\int_Y {\mathcal K}(x, y)\bs u(y)d\mu_x(y),
\end{equation}
where $x$ is an element of a point set $X$, $\mu_x$ is a depending on $x\in X$ finite positive 
regular Borel measure on the Borel $\sigma$-algebra 
of  locally compact Hausdorff space $Y$, $\bs u$ is a real vector-valued function 
with $m$ components which are Borel and bounded on $Y$, kernel 
${\mathcal K}(x, \cdot)$ of the transform is a real $(m\times m)$-matrix-valued 
function with Borel and bounded elements on $Y$ for any $x\in X$. We suppose that 
${\mathcal K}$ is normalized by the condition
\begin{equation} \label{EQ_102}
\int_Y {\mathcal K}(x, y)d\mu_x(y)=I
\end{equation}
for any $x\in X$, where $I$ is the identity $(m\times m)$-matrix.

We say that ${\mathfrak S}$ is invariant for the integral transform 
(\ref{EQ_101}) if $(T\bs u)(x) \in {\mathfrak S}$ for all $x\in X$ and for any bounded and 
Borel real $m$-component vector-valued function $\bs u$ which takes values in ${\mathfrak S}$.

As a simple example of the integral transform for which any interval $[\alpha , \beta]$ 
is invariant, we mention
$$
(Su)(x)=\left (\int_a^b s(x,y)dy\right )^{-1}{\int_a^b s(x,y)u(y)dy},
$$
where $s(x, y)$ is continuous and positive function on the bounded set $[c, d]\times[a, b]$,
$u$ is a continuous function on $[a, b]$. 

Another example of integral transform for which 
any interval $[\alpha , \beta]$ is invariant, is the double layer potential 
$$
(L \varphi)(x)={1\over \omega_n}\int_{{\mathbb R}^n} \varphi(y)\omega_{_D}(x, dy),
$$
where $\omega_n$ is the area of the unit sphere in ${\mathbb R}^n$, $D$ is an arbitrary convex 
bounded domain in ${\mathbb R}^n$, $n\geq 2$, $x\in D$, 
$\varphi$ belongs to the set of continuous functions on ${\mathbb R}^n$ with compact 
support, and 
$$
\omega_{_D}(x, B)=\int_{B\cap \partial D}{(\bs \nu_y, y-x) \over |y-x|^n}\;d\sigma_y
$$ 
is the solid angle at which the intersection of Borel set 
$B\subset {\mathbb R}^n$ and the  boundary $\partial D$ of $D$
is seen from the point $x$. Here $\bs \nu_y$ is the outward unit normal to $\partial D$ 
at the point $y$ (see Burago and Maz'ya \cite{BM}).

\smallskip
In  section 4 we obtain the following necessary and sufficient condition 
on the matrix-valued kernel ${\mathcal K}$ for which ${\mathfrak S}$ is 
invariant for the transform $T$.

\setcounter{theorem}{0}
\begin{proposition} \label{P_1}
A convex body ${\mathfrak S}$ is invariant for transform $(\ref{EQ_101})$ normalized by 
$(\ref{EQ_102})$ if and only if there exists a bounded non-negative function 
$g: X\times Y\times {\mathfrak N}_{\mathfrak S}\rightarrow {\mathbb R}$ such that
\begin{equation} \label{EQ_3}
^{t}{\mathcal K}(x, y)\bs \nu=g(x, y; \bs \nu)\bs \nu
\end{equation}
for almost all $y\in Y$. 
\end{proposition}

\smallskip
In section 5 we consider a strongly elliptic system of the form
\begin{equation} \label{sys}
{\mathfrak A}_0({D_x})\bs u=\sum_{j,k=1}^n {\mathcal A}_{jk}
{\partial ^2 \bs u\over \partial x_j \partial x_k}=\bs 0
\end{equation}
in the half-space ${\mathbb R}^n _+$, 
where ${\mathcal A}_{jk}={\mathcal A}_{kj}$ are real constant $(m\times m)$-matrices.
For this system we obtain the following two criteria for the invariance of 
some convex bodies, where the matrix $A$ is not necessarily symmetric.

\setcounter{theorem}{2}
\begin{theorem} \label{T_3} An orthant ${\mathbf R}^m_+(\alpha_1,\dots\alpha_m)=
\{u=(u_1,\dots , u_m): u_1\geq \alpha_1,\dots , u_m\geq \alpha_m \}$  
in ${\mathbb R}^m$ is invariant for the system ${\mathfrak A}_0({D_x})\bs u=\bs 0$ in 
${\mathbb R}^n_+$ if and only if 
\begin{equation} \label{diag1}
{\mathfrak A}_0({D_x})=A\;{\rm diag}\{L_1({D_x}),\dots, L_m({D_x}) \}\;,
\end{equation}
where 
$$ 
L_i({D_x})=\sum_{j,k=1}^n a_{jk}^{(i)}{\partial ^2 \over \partial x_j \partial x_k}\;,
\;\;\;\;\;\;i=1,\dots,m,
$$
are scalar elliptic operators and $A$ is a non-degenerate $(m\times m)$-matrix 
such that operator $(\ref{diag1})$ is strongly elliptic.
\end{theorem}

\begin{theorem} \label{T_4} Let on the boundary of a convex body 
${\mathfrak S} \subset {\mathbb R}^m$ there exists a set of unit outward normals $\{\bs\nu_1,\dots,
\bs \nu_m, \bs\nu_{m+1} \}$ such that arbitrary $m$ vectors of this collection 
are linear independent.
The body ${\mathfrak S}$ is invariant for the system ${\mathfrak A}_0({D_x})\bs u=\bs 0$ 
in ${\mathbb R}^n_+$ if and only if
\begin{equation} \label{scalar}
{\mathfrak A}_0({D_x})=A\;L({D_x})\;,
\end{equation}
where
$$
L({D_x})=\sum_{j,k=1}^n a_{jk}{\partial ^2 \over \partial x_j \partial x_k}
$$
is a scalar elliptic operator and $A$ is a non-degenerate $(m\times m)$-matrix 
such that operator $(\ref{scalar})$ is strongly elliptic.
\end{theorem}

The proof of necessity in Theorems \ref{T_3} and \ref{T_4} is based on Proposition \ref{P_1} on the 
invariance criterion for normalized matrix-valued integral transforms.

\smallskip
The last assertion generalizes our earlier result \cite{MK} on criteria of 
validity of the classical maximum modulus principle for solutions of 
system (\ref{sys}) in ${\mathbb R}^n _+$. We note that convex polyhedral cones with 
$p>m$ facets, convex cones with smooth guide and convex compact bodies 
with smooth boundary satisfy the condition mentioned in Theorem \ref{T_4}.
Obviously, the matrix $A$ in Theorem \ref{T_4} satisfies the inequality 
$(A\bs\zeta, \bs\zeta )>0$ for any $m$-dimensional vector $\bs\zeta\neq \bs 0$.

\smallskip
The criteria on validity of the componentwise maximum principle 
for linear parabolic system of general form
were obtained in the paper by Otsuka \cite{Ots}. 
In our papers \cite{KM1}-\cite{KM3} and \cite{MK} (see also monograph \cite{KM} 
and references therein) the criteria for validity of other type 
of maximum principles for parabolic systems were established, which are interpreted as
conditions for the invariance of compact convex bodies. Recently, criteria for the 
invariance of any convex body (bounded or unbounded) for linear parabolic systems 
without zero order term in the layer were obtained in \cite{KM4}. 

Maximum principles for weakly coupled elliptic and parabolic systems are considered 
in the books by Protter and Weinberger \cite{PW}, and Walter \cite{WAL} which also contain 
rich bibliographies on this subject.
There exists a wide bibliography on invariant sets
for nonlinear parabolic and elliptic systems with principal part subjected
to various structural conditions such as scalarity, diagonality
and others (see, for instance, papers by Alikakos \cite{ALIK1},
Amann \cite{AM}, Bates \cite{Bates}, Bebernes and Schmitt \cite{BS},
Bebernes, Chueh and Fulks \cite{BCF}, 
Chueh, Conley and Smoller \cite{CCS}, 
Conway, Hoff and Smoller \cite{CHS},
Cosner and Schaefer \cite{CS},
Kuiper \cite{Kuip}, Lemmert \cite{Lem}, 
Redheffer and Walter \cite{RW1,RW2},
Schaefer \cite{Shaef}, Smoller \cite{Smoller},
Weinberger \cite{WEIN1} and references there).

\section{Sufficient conditions for the invariance of convex bodies 
for strongly elliptic systems}

By $[{\rm C}_{\rm b}(\overline{\Omega})]^m$ we mean the space of bounded 
$m$-component vector-valued functions which are continuous in $\overline{\Omega}$.
The notation $[{\rm C}_{\rm b}(\partial \Omega )]^m$ has a similar meaning.
Let $[{\rm C}^{2}(\Omega)]^m$ denote the space of $m$-component vector-valued 
functions with continuous derivatives up to the second order in $\Omega$. 
We omit $m$ in the notations of above function spaces in the case $m=1$.  
Analogously we omit $b$ in the notation of the space of continuous functions 
if $\Omega$ is bounded.

Now we obtain a sufficient condition for 
the invariance of a convex body in ${\mathbb R}^m$ for linear uniformly strongly elliptic 
systems without zero order term in a bounded subdomain of ${\mathbb R}^n $.

\smallskip
{\it Proof of Theorem \ref{T_1}.} We fix a point $a\in \partial ^*{\mathfrak S}$. Let 
$\bs u\in [{\rm C}(\overline{\Omega})]^m \cap
[{\rm C}^{2}(\Omega)]^m$ be a solution of the system 
${\mathfrak A}(x, D_x)\bs u=\bs 0$.
Then ${\mathfrak A}(x, D_x)\bs u_a=\bs 0$, where $\bs u_a=\bs u- \bs a$.
Hence,
\begin{eqnarray*}
& &\sum ^n_{j,k=1}\left ({\mathcal A}_{jk}(x){\partial ^2\bs u_a\over\partial
x_j\partial x_k},\;\bs\nu (a)\right )+
\sum ^n_{j=1}\left ({\mathcal A}_j(x){\partial \bs u_a\over\partial x_j},\;\bs\nu (a) \right )\\
& &\\
& &=\sum ^n_{j,k=1}\left ({\partial ^2\bs u_a\over\partial x_j\partial
x_k},\;^{t}\!\!{\mathcal A}_{jk}(x)\bs\nu (a)\right )+\sum ^n_{j=1}\left 
({\partial\bs u_a\over\partial x_j},\;^{t}\!\!{\mathcal A}_j(x)\bs\nu (a)\right )=0.
\end{eqnarray*}
By the last equality and (\ref{(4.1A)}) we arrive at
\begin{eqnarray*}
& &\sum ^n_{j,k=1}\left ({\partial ^2\bs u_a\over\partial x_j\partial
x_k},\;a_{jk}(x)\bs\nu (a)\right )+
\sum ^n_{j=1}\left ({\partial\bs u_a\over\partial x_j},\;a_j(x)\bs\nu (a)\right )\\
& &\\
& &=\sum ^n_{j,k=1}a_{jk}(x){\partial ^2\over\partial x_j\partial
x_k}(\bs u_a,\bs\nu (a))+\sum ^n_{j=1}a_j(x){\partial \over\partial x_j}(\bs u_a,\bs\nu (a) )=0 .
\end{eqnarray*}
Thus the function $u_a=(\bs u_a,\bs\nu (a))$ satisfies the scalar equation
$$
\sum ^n_{j,k=1}a_{jk}(x){\partial ^2 u_a\over\partial x_j\partial x_k}+
\sum ^n_{j=1}a_j(x){\partial u_a\over\partial x_j}=0.
$$
By (\ref{(0.2)}),
$$
\Biggl (\;\sum_{j,k=1}^{n} \sigma _j \sigma _k \bs\zeta ,\;
^{t}\!\!{\mathcal A}_{jk}(x)\bs\zeta \Biggr ) 
\geq\delta|\bs\sigma|^2|\bs\zeta|^2 
$$
for all $\bs\zeta \in {\mathbb R}^m$, $\bs\sigma \in {\mathbb R}^n$ and any 
$x\in \Omega$. The last inequality with 
$\bs \zeta=\bs\nu$ together with (\ref{(4.1A)}) imply
\begin{equation} \label{uni}
\sum_{j,k=1}^{n}a_{jk}(x)\sigma _j \sigma _k  \geq \delta|\bs\sigma|^2
\end{equation}
for any $x\in \Omega$ and all 
$\bs\sigma  \in {\mathbb R}^n$. Therefore, by the maximum principle
for solutions to the uniformly elliptic equation without zero order term in a bounded 
domain $\Omega$ (see, e.g., Gilbarg and Trudinger \cite{GiTr}, Sect. 3.1) 
with the unknown function $u_a \in {\rm C}(\overline{\Omega}) \cap
{\rm C}^{2}(\Omega)$, we conclude that
$$
\big ( \bs u(x)-\bs a, \bs \nu (a) \big ) \leq
\max_{y\in \partial \Omega}\big (\bs u(y)-\bs a, \bs \nu (a) \big ),\;\;\;\;\;x\in\Omega,
$$
i.e., the half-space ${\mathbb R}^m_{\bs \nu (a)}(\bs a)$ is invariant for the system 
${\mathfrak A}(x, D_x)\bs u=\bs 0$ in $\Omega$.

Using the known equality (Rockafellar \cite{ROCK}, Theorem 18.8):
\begin{equation} \label{Rock}
{\mathfrak S}=\bigcap_{a \in \partial^*{\mathfrak S}}{\mathbb R}^m_{\bs \nu(a)}(\bs a)\;,
\end{equation}
we complete the proof.
\blankbox

{\bf Remark.} Let ${\mathcal A}$ be a bounded non-degenerate $(m\times m)$-matrix-valued 
function in $\Omega$. Since the systems ${\mathfrak A}(x, D_x)\bs u=\bs 0$ and
${\mathcal A}(x){\mathfrak A}(x, D_x)\bs u=\bs 0$ are equivalent, the formulated 
in Theorem \ref{T_1} sufficient condition for the invariance of convex bodies 
for ${\mathfrak A}(x, D_x)\bs u=\bs 0$ also holds for the system
${\mathcal A}(x){\mathfrak A}(x, D_x)\bs u=\bs 0$.

It follows from the proof of Theorem \ref{T_1} that condition (\ref{(0.2)}) 
of uniformly strongly ellipticity of the system ${\mathfrak A}(x, D_x)\bs u=\bs 0$ can be relaxed 
by putting $\bs\zeta \in {\mathfrak N}_{\mathfrak S}$ instead of all $\bs\zeta \in {\mathbb R}^m $.

\medskip
The following assertion of the Phragm\'en-Lindel\"of type 
is borrowed from the book by Landis \cite{LAN} (Theorem 6.3).

\begin{lemma} \label{PhL} Denote by $K_h$ the cone
$$
x_n^2> h^2 \sum_{i=1}^{n-1}x^2_i\;,\;\;\;x_n<0\;.
$$
Let $h>1$. Let $\Omega$ be an unbounded domain and let $K_h$ belong to the
complement of $\Omega$. Let
$$
L(x, D_x)=\sum_{j,k=1}^n a_{jk}(x){\partial ^2 \over \partial x_j \partial x_k}
$$
be a scalar uniformly elliptic operator in $\Omega$ with the ellipticity constant
$$
e=\sup _{x\in \Omega, |\xi|=1}
{\sum_{j=1}^n a_{jj}(x) \over \sum_{j,k=1}^n a_{jk}(x)\xi_j\xi_k}
$$
and let a subelliptic function $u(x)$ be continuous in 
the closure of $\Omega$ and nonpositive on the boundary of $\Omega$.
Then one of the following assertions holds:

{\rm (a)} $u(x)\leq 0$ everywhere in $\Omega$;

{\rm (b)} 
$$
\liminf_{r\rightarrow\infty}\;M(r)/r^{ch^{-s}}>0\;,
$$ 
where
$$
M(r)=\max \{ u(x): x \in {\overline \Omega}, |x|=r \}
$$
and $c>0$ is a constant depending on $e$ and $s=n-2$.
\end{lemma}

Let $u\in {\rm C}_{\rm b}(\overline{\Omega}) \cap
{\rm C}^{2}(\Omega)$ be a solution of $L(x, D_x)u=0$ in 
an unbounded domain $\Omega$ described in Lemma \ref{PhL}.
We introduce the function
$$
w(x)=u(x)-{\overline M},
$$ 
where 
$$
{\overline M}=\sup_{y\in \partial \Omega} u(y).
$$
Since $L(x, D_x)w=0$ in $\Omega$, it follows from Lemma \ref{PhL} that 
$w(x)\leq 0$ everywhere in $\Omega$. 
Thus, for any solution $u\in {\rm C}_{\rm b}(\overline{\Omega}) \cap
{\rm C}^{2}(\Omega)$ of the equation $L(x, D_x)u=0$ the maximum principle
\begin{equation} \label{(MP)}
u(x) \leq \sup_{y\in \partial \Omega}u(y),\;\;\;x \in \Omega,
\end{equation}
holds.

\smallskip
Now, we turn to  

\smallskip
{\it Proof of Theorem \ref{T_2}.} We fix a point $a\in \partial ^*{\mathfrak S}$. 
Let $\bs v\in [{\rm C}_{\rm b}(\overline{\Omega})]^m \cap
[{\rm C}^{2}(\Omega)]^m$ be a solution of system (\ref{Q1}). 
Then ${\mathfrak B}(x, D_x)\bs v_a=\bs 0$, where $\bs v_a=\bs v- \bs a$.

Consider the linear system
$$
\sum ^n_{j,k=1}{\mathcal A}_{jk}(x)
{\partial ^2 \bs u \over\partial x_j\partial x_k}=\bs 0, 
$$
where ${\mathcal A}_{jk}(x)={\mathcal B}_{jk}(x, D_x \bs v_a(x))$ and 
$\bs u\in [{\rm C}_{\rm b}(\overline{\Omega})]^m 
\cap [{\rm C}^{2}(\Omega)]^m$ is an unknown vector-valued function.
In particular, the last system has the solution $\bs u=\bs v_a$.

Putting ${\mathcal A}_1=\dots={\mathcal A}_n=0$ in the proof of Theorem 1
and using the maximum principle (\ref{(MP)}) for the scalar uniformly elliptic 
equation $L(x, D_x)u=0$ in an unbounded domain $\Omega$ described in Lemma \ref{PhL}, we 
arrive at Theorem \ref{T_2}.
\blankbox 

\section{Matrices subject to (\ref{(Q4)}) for certain convex bodies}

We say that an $(m\times m)$-matrix ${\mathcal A}$ satisfies condition (\ref{(Q4)})
for a convex body ${\mathfrak S}$ if condition (\ref{(Q4)}) holds for 
any $\bs\nu \in {\mathfrak N}_{\mathfrak S}$. 
In this section we describe the structure of matrices ${\mathcal A}$ which satisfy
(\ref{(Q4)}) for certain classes of convex bodies.

\medskip
{\bf Polyhedral angles.} We introduce the polyhedral angle 
$$
{\mathbf R}^m_+(\alpha_{m-k+1}, \dots,\alpha_m)=\{ u=(u_1,\dots, u_m): u_{m-k+1}\geq \alpha_{m-k+1}, \dots,u_m\geq \alpha_m \},
$$
where $k=1,\dots, m$. In particular, ${\mathbf R}^m_+(\alpha_m)$ is a half-space, 
${\mathbf R}^m_+(\alpha_{m-1},\alpha_m)$ is a dihedral angle, and ${\mathbf R}^m_+(\alpha_1,\dots\alpha_m)$ 
is an orthant in ${\mathbb R}^m$. 

\smallskip
\begin{lemma} \label{1A} A matrix ${\mathcal A}$ of order $m$ satisfies  
$(\ref{(Q4)})$ for the polyhedral angle ${\mathbf R}^m_+(\alpha_{m-k+1}, \dots,\alpha_m)$ 
if and only if all nondiagonal elements of  $m-k+1$-th,\dots, $m$-th rows of ${\mathcal A}$ are equal to zero.

In particular, a matrix ${\mathcal A}$ of the second order satisfies $(\ref{(Q4)})$
for the half-plane ${\mathbf R}^2_+(\alpha_2)$ if and only if ${\mathcal A}$ is upper triangular.
\end{lemma}
\begin{proof} Let $\bs e_j$ stand for the unit vector of the $j$-th coordinate axis. The vectors 
$\bs \nu_{m-k+1}=-\bs e_{m-k+1},\dots, \bs \nu_m=-\bs e_m$ form the family of unit outward normals to
${\mathbf R}^m_+(\alpha_{m-k+1}, \dots,\alpha_m)$.

By  ${\mathcal A}^{(j)}$ we denote the $j$-th row of the matrix ${\mathcal A}$. Let 
${\mathcal A}=((a_{j,k}))$ satisfy (\ref{(Q4)}) for ${\mathbf R}^m_+(\alpha_{m-k+1}, \dots,\alpha_m)$.
Then for any $j=m-k+1, \dots , m$ we have
\begin{equation} \label{U_1}
^t\!\!{\mathcal A}\bs \nu_j=a(\bs \nu_j)\bs \nu_j\;,
\end{equation}
i.e.,
$$
^t\!\!{\mathcal A}^{(j)}=a(-\bs e_j)\bs e_j\;.
$$
Hence, all elements of the column $^t\!\!{\mathcal A}^{(j)}$, except for $j$-th one, are equal to zero.

Conversely, let all nondiagonal elements of the $m-k+1$-th,\dots, $m$-th rows of ${\mathcal A}$ equal zero.
Then (\ref{U_1}) holds with $a(\bs \nu_j)=a_{j,j}$, $j=m-k+1, \dots , m$,
i.e., ${\mathcal A}$ is subject to (\ref{(Q4)}) for ${\mathbf R}^m_+(\alpha_{m-k+1}, \dots,\alpha_m)$.
\end{proof}
  
\medskip 
{\bf Cylinders.} Let 
$$
{\mathbf R}^m_-(\beta_{m-k+1}, \dots,\beta_m)=\{ u=(u_1,\dots, u_m): 
u_{m-k+1}\leq \beta_{m-k+1}, \dots,u_m\leq \beta_m \}
$$
be a polyhedral angle and $\alpha_{m-k+1}<\beta_{m-k+1}, \dots, \alpha_m<\beta_m$.

Let us introduce a polyhedral cylinder 
$$
{\mathbf C}^m(\alpha_{m-k+1}, \dots,\alpha_m; \beta_{m-k+1}, \dots,\beta_m)=
{\mathbf R}^m_+(\alpha_{m-k+1}, \dots,\alpha_m)\cap{\mathbf R}^m_-(\beta_{m-k+1}, 
\dots,\beta_m),\; k<m.
$$

In particular, ${\mathbf C}^m(\alpha_m; \beta_m)$ is a layer and 
${\mathbf C}^m(\alpha_{m-1},\alpha_m; \beta_{m-1},\beta_m)$ is a rectangular cylinder.

\smallskip
Since the collection of unit outward normals to polyhedral cylinder
$$
{\mathbf C}^m(\alpha_{m-k+1}, \dots,\alpha_m; \beta_{m-k+1}, \dots,\beta_m)
$$
consists of the vectors $\bs e_{m-k+1}, -\bs e_{m-k+1}\dots, \bs e_m, -\bs e_m$, the next auxiliary assertion
can be proved similarly to Lemma \ref{1A}.

\smallskip
\begin{lemma} \label{2A} A matrix ${\mathcal A}$ of order $m$ satisfies 
$(\ref{(Q4)})$ for the polyhedral cylinder 
$$
{\mathbf C}^m(\alpha_{m-k+1}, 
\dots,\alpha_m; \beta_{m-k+1}, \dots,\beta_m)
$$ 
if and only if all nondiagonal elements of $m-k+1$-th, $m-k+2$-th,\dots, $m$-th 
rows of ${\mathcal A}$ are equal to zero. 

In particular, a matrix ${\mathcal A}$ of the second order satisfies $(\ref{(Q4)})$
for a strip ${\mathbf C}^2(\alpha_2; \beta_2)$ if and only if ${\mathcal A}$
is upper triangular.
\end{lemma}

\smallskip
Let us introduce the body
$$
{\mathbf S}^m_k(R)=\{ u=(u_1,\dots, u_m): u_{m-k+1}^2+\dots +u_m^2\leq R^2 \},\;\;\;m\geq 3,
$$
which is a spherical cylinder for $k=2,\dots, m-1$. 

\smallskip 
\begin{lemma} \label{3A} A matrix ${\mathcal A}$ of order $m$ satisfies 
$(\ref{(Q4)})$ for the body ${\mathbf S}^m_k(R)$ if and only if:
 
{\rm (i)} all nondiagonal elements of $m-k+1$-th, $m-k+2$-th,\dots, $m$-th rows of 
${\mathcal A}$ are equal to zero; 

{\rm (ii)} all $m-k+1$-th, $m-k+2$-th,\dots, $m$-th diagonal elements of 
${\mathcal A}$ are equal.
\end{lemma}
\begin{proof} The set of unit outward normals to the cylinder ${\mathbf S}^m_k(R)$
is formed by the $m$-dimensional vectors 
\begin{equation} \label{U_2}
(0, \dots , 0,\gamma_{m-k+1},\dots ,\gamma_m)\;,
\end{equation}
where $\gamma^2_{m-k+1}+\dots +\gamma^2_m=1$.

Let the matrix ${\mathcal A}=((a_{j,k}))$ satisfy (\ref{(Q4)}) for the cylinder ${\mathbf S}^m_k(R)$.
The vectors $\bs e_{m-k+1}, \dots, \bs e_m$ are contained in the set of unit outward normals to ${\mathbf S}^m_k(R)$.
Therefore, the necessity of condition (i) in the present Lemma is established in the same way as in Lemma \ref{1A}.
By $\bs\nu_*$ we denote the unit outward normal to ${\mathbf S}^m_k(R)$ with
$$
\gamma_{m-k+1}=\dots =\gamma_m={1\over \sqrt{k}}\;.
$$
Since $^t\!\!{\mathcal A}\bs\nu_*=a(\bs\nu_*)\bs\nu_*$, it follows that
$$
a_{j,j}=a(\bs\nu_*),\;\;\;j=m-k+1,\dots , m.
$$
The necessity of (ii) follows.

Conversely, if the matrix ${\mathcal A}$ has the structure, described in (i), (ii) and $a_{m-k+1,m-k+1}=\dots=a_{m,m}=a$,
then it satisfies (\ref{(Q4)}) for all unit vectors of the form (\ref{U_2}) with $a(\bs \nu)=a$.
\end{proof}  

\medskip
{\bf Cones.} By ${\mathbf K}^m_p$ we denote a convex polyhedral 
cone in ${\mathbb R}^m$ with $p$ facets. Let, further, $\{\bs\nu_1,\dots, \bs\nu_p \}$ be 
the set of unit outward normals to the facets of this cone. 
By $[\bs v_1,\dots, \bs v_m]$ we mean the
$(m\times m)$-matrix whose columns are $m$-component vectors $\bs v_1,\dots, \bs v_m$.

\smallskip
We give an auxiliary assertion of geometric character. 

\setcounter{theorem}{0}
\begin{lemma} \label{L_2} Let $p\geq m$. Then any system $\bs\nu_1,\dots,\bs\nu_m $ 
of unit outward normals to $m$ different facets of ${\mathbf K}^m_p$ is linear independent.
\end{lemma}
\begin{proof} 
By $F_i$ we denote the facet of ${\mathbf K}^m_p$ for which the 
vector $\bs\nu_i$ is normal, $1\leq i\leq m$. Let $T_i$ be the supporting 
plane of this facet. We place the origin
of the coordinate system with the orthonormal basis $\bs e_1,\dots ,\bs e_m$ at
an interior point ${\mathcal O}$ of ${\mathbf K}^m_p$ and use the notation 
$x={\mathcal O}q$, where $q$ is the vertex of the cone. Further, 
let $d_i=\mbox{dist}\;({\mathcal O}, F_i),\;i=1,\dots ,m$. Since
$$
q=\bigcap ^m_{i=1}T_i\;,
$$
it follows that $x=(x_1,\dots ,x_m)$ is the only solution of the system
$$
(\bs\nu_i ,x)=d_i,\;\;i=1,2,\dots ,m,
$$
or, which is the same,
\[
\sum ^m_{j=1}(\bs\nu_i ,\bs e_j)x_j=d_i,\;\;i=1,2,\dots ,m.
\]
The matrix of this system is $^{t}[\bs\nu_1,\dots ,\bs\nu_m]$. Consequently, 
$$
\det {^{t}[\bs\nu_1,\dots ,\bs\nu_m]}\not= 0. 
$$
This implies the linear independence of the system $\bs\nu_1,\dots ,\bs\nu_m$.
\end{proof}

\smallskip 
\begin{lemma} \label{4A} Let there exist a system of unit outward normals $\{\bs\nu_1,\dots,\bs \nu_m, \bs\nu \}$ 
on the boundary of a convex body ${\mathfrak S}$ such that arbitrary $m$ vectors of this system are linear independent.
A matrix ${\mathcal A}$ of order $m$ satisfies $(\ref{(Q4)})$ for the 
body ${\mathfrak S}$ if and only if ${\mathcal A}$ is scalar.
\end{lemma}
\begin{proof} By assumption, arbitrary $m$ vectors in the collection 
$\{\bs\nu_1,\dots,\bs \nu_m, \bs\nu \}$ are linear independent.
Hence there are no zero coefficients $\gamma_i$ in the representation
$\bs \nu=\gamma_1\bs\nu_1+\dots+\gamma_m\bs\nu_m$.

Let (\ref{(Q4)}) hold. 
Then
\begin{equation} \label{K_6}
^{t}{\mathcal A}\bs \nu=a\bs \nu,\;^{t}{\mathcal A}\bs \nu_1=a_1\bs \nu_1, 
\dots,\; ^{t}{\mathcal A}\bs \nu_m=a_m\bs \nu_m\;,
\end{equation} 
where $a, a_1,\dots, a_m$ are scalars. Therefore,
$$
a\sum_{i=1}^m\gamma_i \bs\nu_i=a\bs\nu=\;^{t}{\mathcal A}\bs \nu=\; ^{t}{\mathcal A}
\sum_{i=1}^m\gamma_i \bs\nu_i=\sum_{i=1}^m\gamma_i a_i\bs\nu_i.
$$
Thus,
$$
\sum_{i=1}^m (a-a_i)\gamma_i \bs\nu_i=\bs 0. 
$$
Hence, $a_i=a$ for $i=1,\dots,m$ and consequently ${\mathcal A}$ is a scalar matrix.

Conversely, if ${\mathcal A}=a\;\mbox{diag}\;\{1 ,\dots, 1 \}$, then (\ref{(Q4)}) 
with $a(\gamma)=a$ holds  for ${\mathfrak S}$.
\end{proof} 

\smallskip 
\begin{lemma} \label{5A} A matrix ${\mathcal A}$ of order $m$ satisfies $(\ref{(Q4)})$ for 
the convex polyhedral cone ${\mathbf K}^m_m$ if and only if 
\begin{equation} \label{K_1}
{\mathcal A}=\big (\;^{t}[\bs\nu_1,\dots, \bs\nu_m] \big)^{-1}\;
{\mathcal D}\;^{t}[\bs\nu_1,\dots, \bs\nu_m]\;,
\end{equation}
where ${\mathcal D}$ is diagonal.

A matrix ${\mathcal A}$ of order $m$ satisfies $(\ref{(Q4)})$ either
for the convex polyhedral cone ${\mathbf K}^m_p$ with $p>m$ or for any convex cone with a smooth guide
if and only if ${\mathcal A}$ is scalar.
\end{lemma}
\begin{proof} (i) If ${\mathfrak S}={\mathbf K}^m_m$, we write (\ref{(Q4)}) as
\begin{equation} \label{K_4}
^{t}{\mathcal A}\bs \nu_1=a_1\bs \nu_1, 
\dots,\;^{t}{\mathcal A}\bs \nu_m=a_m\bs \nu_m\;,
\end{equation}
where $\{\bs\nu_1,\dots,\bs \nu_m \}$ is the set of unit outward normals 
to the facets of ${\mathbf K}^m_m$. These normals are linear independent 
by Lemma \ref{L_2}. Let ${\mathcal D}=\mbox{diag}\;\{ a_1,\dots,a_m \}$. 
Equations (\ref{K_4}) can be written as
$$
^{t}{\mathcal A}[\bs \nu_1,\dots,\bs \nu_m]=[\bs \nu_1,\dots,\bs \nu_m]\;{\mathcal D},
$$
which leads to the representation
\begin{equation} \label{K_5}
{\mathcal A}=\big (\;^{t}[\bs \nu_1,\dots,\bs \nu_m]\big )^{-1}\;{\mathcal D}\;
^{t}[\bs \nu_1,\dots,\bs \nu_m]\;.
\end{equation}
Now, (\ref{K_5}) is equivalent to (\ref{K_1}). 

\medskip
(ii) Let us consider the cone ${\mathbf K}^m_p$ with $p>m$. 
By $\{\bs\nu_1,\dots,\bs \nu_m \}$ we denote a system of unit outward normals to
$m$ facets of ${\mathbf K}^m_p$. Let also $\bs \nu$ be a normal to a certain $m\!+\!1$-th facet.
By Lemma \ref{L_2}, arbitrary $m$ vectors in the collection 
$\{\bs\nu_1,\dots,\bs \nu_m, \bs\nu \}$ are linear independent.
Using Assertion 4, we complete the proof for the case $p>m$.

\medskip
(iii) Let (\ref{(Q4)}) hold for the cone ${\mathbf K}$ with a smooth guide.
This cone ${\mathbf K}$ can be inscribed into a polyhedral cone 
${\mathbf K}^m_{m+1}$. Let $\{\bs\nu_1,\dots,\bs \nu_m, \bs\nu \}$ be a system 
of unit outward normals to the facets of ${\mathbf K}^m_{m+1}$. 
This system is a subset of the collection of normals to the boundary of ${\mathbf K}$. 
By Lemma \ref{L_2}, arbitrary $m$ vectors in the set 
$\{\bs\nu_1,\dots,\bs \nu_m, \bs\nu \}$ are linear independent.
Repeating word by word the argument used in (ii) we arrive at the
scalarity of ${\mathcal A}$.

Conversely, (\ref{(Q4)}) is an obvious consequence of the scalarity of 
${\mathcal A}$ for ${\mathfrak S}={\mathbf K}$.

The proof is complete.
\end{proof}
 
\smallskip
Let us consider condition (\ref{(Q4)}) in the case $m=3$.
Using notation and Lemmas 2-4,6, 7 we obtain the following statements.

\smallskip
(i) {\it A matrix ${\mathcal A}$ satisfies $(\ref{(Q4)})$ for the
half-space ${\mathbf R}^3_+(\alpha_3)=\{\bs u=(u_1,u_2, u_3): u_3\geq \alpha_3 \}$ and the layer
${\mathbf C}^3(\alpha_3; \beta_3)=\{\bs u=(u_1,u_2, u_3): \alpha_3\leq u_3\leq \beta_3 \}$
if and only if all non-diagonal elements of the third 
row of ${\mathcal A}$ are equal to zero.}

\smallskip
(ii) {\it A matrix ${\mathcal A}$ satisfies $(\ref{(Q4)})$ for the
dihedral angle ${\mathbf R}^3_+(\alpha_2, \alpha_3)=\{\bs u=(u_1,u_2, u_3): u_2\geq \alpha_2, 
u_3\geq \alpha_3 \}$ and the rectangular cylinder ${\mathbf C}^3(\alpha_2, \alpha_3; \beta_2, \beta_3)=
\{\bs u=(u_1,u_2, u_3): \alpha_2 \leq u_2\leq \beta_2, \alpha_3\leq u_3\leq \beta_3 \}$
if and only if all non-diagonal elements of the second and third 
rows of ${\mathcal A}$ are equal to zero.}

\smallskip
(iii) {\it A matrix ${\mathcal A}$ satisfies $(\ref{(Q4)})$ for the
orthant ${\mathbf R}^3_+(\alpha_1, \alpha_2, \alpha_3)=\{\bs u=(u_1,u_2, u_3): u_1\geq \alpha_1, u_2\geq \alpha_2, u_3\geq \alpha_3 \}$ and the parallelepiped ${\mathbf C}^3(\alpha_1, \alpha_2, \alpha_3; \beta_1, \beta_2, \beta_3)=\{\bs u=(u_1,u_2, u_3): \alpha_1 \leq u_1\leq \beta_1,, \alpha_2 \leq u_2\leq \beta_2, \alpha_3\leq u_3\leq \beta_3 \}$ if and 
only if ${\mathcal A}$ is diagonal.}

\smallskip
(iv) {\it A matrix ${\mathcal A}$ satisfies $(\ref{(Q4)})$ for the
circular cylinder ${\mathbf S}^3_2(R)=\{\bs u=(u_1,u_2, u_3): u_2^2+u_3^2\leq R^2 \}$
if and only if all non-diagonal elements of the second and third 
rows of ${\mathcal A}$ are equal to zero and the diagonal elements of the 
same rows are equal.}

\smallskip
(v) {\it A matrix ${\mathcal A}$ satisfies $(\ref{(Q4)})$ for the
three-hedral cone  ${\mathbf K}^3_3$ with unit outward normals $\bs\nu_1, \bs\nu_2, \bs\nu_3$ to
their facets if and only if 
$$
{\mathcal A}=\big (\;^{t}[\bs\nu_1, \bs\nu_2, \bs\nu_3] \big)^{-1}\;
{\mathcal D}\;^{t}[\bs\nu_1, \bs\nu_2, \bs\nu_3],
$$
where ${\mathcal D}$ is diagonal.}

\smallskip
(vi) {\it A matrix ${\mathcal A}$ satisfies $(\ref{(Q4)})$ either for the
convex polyhedral cone  ${\mathbf K}^3_p$ with $p$ facets, $p>3$, or for any convex 
cone with a smooth guide or for an arbitrary compact convex body with smooth boundary if and only if 
${\mathcal A}$ is scalar.}

\section{Criterion for the invariance of convex bodies for normalized matrix-valued integral transforms}

Let $\bs\nu$ be a fixed $m$-dimensional unit vector, let 
$\bs a$ be a fixed $m$-dimensional vector, and let 
${\mathbb R}^m_{\bs \nu}(\bs a)=\{\bs u\in {\mathbb R}^m: (\bs u-\bs a, \bs \nu)\leq 0 \}$.

\smallskip
Now we obtain a necessary and sufficient condition on 
the matrix-valued kernel ${\mathcal K}$ for which ${\mathfrak S}$ is invariant 
for the integral transform $T$ defined by 
(\ref{EQ_101}) and normalized by (\ref{EQ_102}).

\smallskip
{\it Proof of Proposition \ref{P_1}.} 
(i) {\sl Necessity.} Suppose  that ${\mathfrak S}$ 
is invariant for $T$. Let $x\in X$ be fixed. We take a point $a\in \partial^*{\mathfrak S}$ 
and denote $\bs\nu(a)$ by $\bs\nu$. 

By (\ref{EQ_102}), we have  
\begin{equation} \label{EQ_4}
(T\bs u)(x)-\bs a=\int_Y {\mathcal K}(x, y)\big (\bs u(y)-\bs a \big )d\mu_x(y).
\end{equation}

We represent $^{t}{\mathcal K}(x, y)\bs \nu$ as
\begin{equation} \label{(3.7A)}
^{t}{\mathcal K}(x, y)\bs \nu=g(x, y; \bs \nu)\bs \nu+\bs f(x, y; \bs \nu),
\end{equation}
where 
\begin{equation} \label{(3.K)}
g(x, y; \bs \nu)=\big (\; ^{t}{\mathcal K}(x, y)\bs \nu, \bs \nu \big )
\end{equation}
and
\begin{equation} \label{(3.S)}
\bs f(x, y; \bs \nu)=\; ^{t}{\mathcal K}(x, y) \bs \nu 
-\big ( \;^{t}{\mathcal K}(x, y)\bs \nu, \bs \nu \big )\bs\nu .
\end{equation}

Suppose there exists a set ${\mathcal M}\subset Y$ with $\mu_x({\mathcal M})>0 $
such that for all $y \in {\mathcal M}$, the inequality
\begin{equation} \label{(3.SS)}
\bs f(x, y; \bs \nu)\neq \bs 0
\end{equation}
holds, and for all $y \in Y\backslash {\mathcal M}$ the equality
$\bs f(x, y; \bs \nu)= \bs 0$ is valid. 

Further, we set 
\begin{equation} \label{(3.8A)}
\bs u(y)- \bs a=\alpha\bs f(x, y; \bs \nu)-\beta\bs\nu ,
\end{equation}
where $\alpha>0$, $\beta \geq 0$.
It follows from (\ref{(3.S)}) and (\ref{(3.8A)}) that
\begin{equation} \label{(3.8AY)} 
(\bs u(y)- \bs a, \bs \nu )=-\beta\leq 0,\;\;\;\;\;\;
|\bs u(y)- \bs a|=\big ( \alpha^2|\bs f(x, y; \bs \nu)|^2+\beta^2 \big )^{1/2}
\end{equation}
and
\begin{equation} \label{(3.8AU)}
(\bs u(y)- \bs a,\; ^{t}{\mathcal K}(x, y) \bs \nu)=\alpha|\bs f(x, y; \bs \nu)|^2 - 
\beta \big (\; ^{t}{\mathcal K}(x, y)\bs \nu, \bs \nu \big ).
\end{equation}

We introduce a Cartesian coordinate system ${\mathcal O}\xi_1\dots{\mathcal O}\xi_{m-1}$
in the hyperplane, tangent to $\partial {\mathfrak S}$ with the origin at the point ${\mathcal O}=a$.
We direct the axis ${\mathcal O}\xi_m$ along the interior normal to $\partial {\mathfrak S}$.
Let $\bs e_1, \dots, \bs e_m$ denote the coordinate orthonormal basis of this system and let 
$\xi '=(\xi_1,\dots, \xi_{m-1})$.

We use the notation 
$$
\lambda =\sup \{ |\bs f(x, y; \bs \nu)|: y \in Y \}.
$$
Let $\partial {\mathfrak S}$ be described by the equation $\xi_m=F(\xi')$ in a 
neighbourhood of ${\mathcal O}$, where $F$ is convex and differentiable at ${\mathcal O}$. 

We put $\beta=\max\;\{ F(\xi'): |\xi'|=\alpha\lambda \}$. By (\ref{(3.8AY)}), 
$$ 
(\bs u(y)- \bs a, \bs e_m )=\beta \geq 0,\;\;\;\;|\bs u(y)- \bs a|\leq (\alpha^2\lambda^2+\beta^2)^{1/2},
$$
which implies $\bs u(y) \in {\mathfrak S}$ for all $y \in Y$.

By invariance of the convex body ${\mathfrak S}$, this gives 
\begin{eqnarray} \label{(3.5A)}
\big ( (T\bs u)(x)-\bs a, \bs \nu \big )&=&
\int _Y\left ( {\mathcal K}(x, y) \big ( \bs u(y)- \bs a \big ), \bs \nu\right ) d\mu_x(y)\nonumber\\
& &\nonumber\\
& &\\
&=&\int _Y\left (  \bs u (y)- \bs a ,\; ^{t}{\mathcal K}(x, y)\bs \nu\right ) d\mu_x(y) \leq 0.\nonumber
\end{eqnarray}
Now, by (\ref{(3.5A)}) and (\ref{(3.8AU)}),  
$$
0\geq\big ( (T\bs u)(x)-\bs a, \bs \nu \big )=\int _Y \big [ \alpha|\bs f(x, y; \bs \nu)|^2 
- \beta \big (\; ^{t}{\mathcal K}(x, y)\bs \nu, \bs \nu \big )\big ]d\mu_x(y), 
$$
which along with (\ref{EQ_102}) leads to
\begin{equation} \label{(3.3P)}
0\geq\big ((T\bs u)(x)-\bs a, \bs \nu \big )=\alpha\left ( 
\int _{\mathcal M} |\bs f(x, y; \bs \nu)|^2 d\mu_x(y) - {\beta \over \alpha}\right ).
\end{equation}

By differentiability of $F$ at ${\mathcal O}$, we have $\beta /\alpha\rightarrow 0$ 
as $\alpha \rightarrow 0$. Consequently, one can choose $\alpha$ so small that the 
second factor on the right-hand side of (\ref{(3.3P)}) becomes positive, which 
contradicts the assumption $\mu_x({\mathcal M})>0$. Therefore, 
$ \bs f(x, y; \bs \nu)=\bs 0$ for almost all $y \in Y$. 

Since $x \in X$ and $a\in \partial^*{\mathfrak S}$ are arbitrary,
we arrive at (\ref{EQ_3}) by (\ref{(3.7A)}). 

Now we show that $g(x, y; \bs \nu)\geq 0$ for any $x\in X$, 
$\bs\nu \in {\mathfrak N}_{\mathfrak S}$ and  
almost all $y\in Y$. Suppose that there exist points $x\in X$ and $a \in \partial^*{\mathfrak S}$
such that $g(x, y; \bs \nu)< 0$ on the set ${\mathcal S}\subset Y$ with $\mu_x({\mathcal S})>0$. 
We choose the vector-valued function $\bs u(y) \in {\mathfrak S}$, $y\in Y$, such that 
$-\varepsilon\leq \big (\bs u(y)-\bs a, \bs \nu \big ) <0$ with $\varepsilon>0$  
for $y\in Y\backslash {\mathcal S}$ and 
$\big (\bs u(y)-\bs a, \bs \nu \big )=-1$ for $y\in {\mathcal S}$. Then, by (\ref{EQ_3}),
\begin{eqnarray*} 
\big ( (T\bs u)(x)-\bs a, \bs \nu \big )&\!\!\!=\!\!\!&
\int _Y\left ( {\mathcal K}(x, y) \big ( \bs u(y)- \bs a \big ), \bs \nu\right ) d\mu_x(y)\\
&\!\!\!=\!\!\!&\int _{\mathcal S}g(x, y; \bs \nu)\left (  \bs u (y)- \bs a , \bs \nu\right ) d\mu_x(y)+
\int _{Y\backslash {\mathcal S}}g(x, y; \bs \nu)\left (  \bs u (y)- \bs a , \bs \nu\right ) d\mu_x(y) ,
\end{eqnarray*}
which will be positive for sufficiently small $\varepsilon$, and this contradicts to the invariance of ${\mathfrak S}$.
Therefore, $\mu_x({\mathcal S})=0$.

\smallskip
(ii) {\sl Sufficiency.} Suppose that (\ref{EQ_3}) holds with a non-negative 
$g(x, y; \bs \nu)$ for any $x\in X$, $\bs\nu \in {\mathfrak N}_{\mathfrak S}$ 
and almost all $y\in Y$. We choose a point $a\in \partial^*{\mathfrak S}$ and 
fix a point $x\in X$. Let $\bs u(y) \in {\mathfrak S}$ for any $y\in Y$. Then 
$\left (  \bs u (y)- \bs a , \bs \nu\right )\leq 0$ for $y\in Y$, and therefore
\begin{eqnarray*} 
\big ( (T\bs u)(x)-\bs a, \bs \nu \big )&=&
\int _Y\left ( {\mathcal K}(x, y) \big ( \bs u(y)- \bs a \big ), \bs \nu\right ) d\mu_x(y)\\
&=&\int _Y g(x, y; \bs \nu)\left (  \bs u (y)- \bs a , \bs \nu\right ) d\mu_x(y)\leq 0.
\end{eqnarray*}
Hence, $(T\bs u)(x)-\bs a \in {\mathbb R}^m_{\bs \nu}(\bs a)$. This, by arbitrariness 
of $x\in X$ and $a\in \partial^*{\mathfrak S}$, and representation (\ref{Rock}) of the convex body ${\mathfrak S}$
in ${\mathbb R}^m$, proves the sufficiency.
\blankbox

\section{Criteria for the invariance of some convex bodies for strongly elliptic systems} 

According to Shapiro \cite{SHA} (see also Lopatinski\v{\i} \cite{LO1})  
there exists a bounded solution of the problem
\begin{equation} \label{EELMSR_0.02}
{\mathfrak A}_0(D_x)\bs u=\bs 0\; \hbox {in}\; {\mathbb R}^n_{+},\;\; \bs u=\bs f
\; \hbox {on}\; \partial {\mathbb R}^n _{+},
\end{equation}
with $\bs f\in [{\rm C}_{\rm b}(\partial {\mathbb R}^n _{+})]^m$, such that $\bs u$ is continuous
up to $\partial {\mathbb R}^n _{+}$, and can be represented in the form
\begin{equation} \label{EELMSR_0.03}
\bs u(x)=\int _{\partial {\mathbb R}^n_{+}} {\mathcal M}  \left ({y-x\over |y-x|}\right )
{x_n\over |y-x|^n}\bs f(y')dy'.
\end{equation}
Here $y=(y', 0),\; y'=(y_1,\dots,y_{n-1}),$ and ${\mathcal M}$ is a continuous
$(m\times m)$-matrix-valued function on the closure of the hemisphere
${\mathbb S}^{n-1} _{-}=\big \lbrace x\in {\mathbb R}^n :\;|x|=1, x_{n} < 0 \big \rbrace$
such that
\begin{equation} \label{EELMSR_0.03A}
\int _{{\mathbb S}^{n-1} _ {-}}{\mathcal M}(\sigma ) d \sigma=I.
\end{equation}
Here, as before, $I$ mean the identity $(m\times m)$-matrix.

We note that equality (\ref{EELMSR_0.03}) can be represented in the form
\begin{equation} \label{EELMSR_0.04}
\bs u(x)=\int _{\partial {\mathbb R}^n_{+}}{\mathcal M}\left ({y-x\over |y-x|}\right )
\bs f(y')\; \omega (x, dy') ,
\end{equation}
where
\[
\omega  (x, {\mathcal B})=\int _ {\mathcal B} \;{x_n\over |y-x|^n} dy'
\]
is the solid angle  at which a Borel set
${\cal B} \subset \partial {\mathbb R}^n_{+}$ is seen from the point $x \in {\mathbb R}^n_{+}$.
The solid angle $\omega  (x, \cdot )$ is a finite regular Borel measure
on $\partial {\mathbb R}^n_{+}$ for any fixed $x \in {\mathbb R}^n_{+}$, and $\omega  (x, 
{\mathcal B}) \geq 0$.

The uniqueness of a solution of the Dirichlet problem (\ref{EELMSR_0.02}) in the
class $[{\rm C}_{\rm b}(\overline {{\mathbb R}^n _{+}})]^m \cap [{\rm C}^{2}({\mathbb R}^n _{+})]^m$ 
can be derived by means of standard arguments from (\ref{EELMSR_0.03}) and from local estimates
of derivatives of solutions to elliptic systems (see Agmon, Douglis and 
Nirenberg \cite{ADN2}, Solonnikov \cite{SOL}).

\smallskip
We note, that analog of Proposition \ref{P_1} can be proved 
almost word by word for the transform (\ref{EELMSR_0.03}) normalized by 
(\ref{EELMSR_0.03A}) with continuous matrix-valued kernel, 
defined on the space of bounded and continuous vector-valued functions. 
The only difference in the formulations of similar statements is that the 
word ''almost'' disappears. In view of this remark, we shall further refer to 
Proposition \ref{P_1} for the transform (\ref{EELMSR_0.03}).

\medskip 
{\it Proof of Theorem $\ref{T_3}$.}  (i) {\sl Necessity.} Let an orthant 
${\mathbf R}^m_+(\alpha_1,\dots\alpha_m)$ in ${\mathbb R}^m$ be 
invariant for the system ${\mathfrak A}_0({D_x})\bs u=\bs 0$ in ${\mathbb R}^n_+$.
Applying Proposition \ref{P_1} to representation (\ref{EELMSR_0.03}) and using Lemma \ref{1A},
we conclude that
a unique solution of the Dirichlet problem (\ref{EELMSR_0.02}) is given by
\begin{equation} \label{EELMSR_0.10}
\bs u(x)=\int _{\partial {\mathbb R}^n_{+}} {\mathfrak D}\left ({y-x\over |y-x|}\right )
{x_n\over |y-x|^n}\bs f(y')dy'\;,
\end{equation}
where ${\mathfrak D}$ is a diagonal $(m\times m)$-matrix-valued kernel with the elements 
${\mathfrak D}_1,\dots,{\mathfrak D}_m$ on the main diagonal.

Let $f_0$ be a scalar function
that is continuous and bounded on $\partial {\mathbb R}^n_{+}$, 
and let
$$
\bs u_s(x)=\int _{\partial {\mathbb R}^n_{+}} {\mathfrak D} \left ({y-x\over |y-x|}\right )
{x_n\over |y-x|^n}\;{\bf c}_s f_0(y')dy',\;\;\;\;\;s=1,\dots,m,
$$
where $\bs c_1=(1,\dots,0), \dots, \bs c_m=(0,\dots,1)$. We denote
\begin{equation} \label{EELMSR_0.11}
u_s(x)=\int _{\partial {\mathbb R}^n_{+}} {\mathfrak D}_s \left ({y-x\over |y-x|}\right )
{x_n\over |y-x|^n}\; f_0(y')dy'\;.
\end{equation}

\medskip
According to (\ref{EELMSR_0.10}), the vector-valued function $\bs u_s$ is a 
solution of the boundary value problem ${\mathfrak A} _0(D_x)\bs u_s =0$ in 
${\mathbb R}^n_{+}, \;\bs u_s=\bs c_s f_0$ on $\partial {\mathbb R}^n_{+}$.
Let $s$ be fixed. Setting $\bs u_s$ instead of $\bs u$ in ${\mathfrak A} _0
(D_x)\bs u =0$, we get the following $m$ boundary value problems
\[
{\mathfrak A} _{is} (D_x)u_s =0 \;\; \hbox {in}\;\;{\mathbb R}^n_{+},\; \;\;
u_s=f_0\;\; \hbox {on}\;\;\partial {\mathbb R}^n_{+},\;\;\; i=1,2\dots,m.
\]
Here ${\mathfrak A} _{is} (D_x)$  is a scalar differential operator
\[
\sum ^n_{j,k=1}{\mathcal A}_{jk}^{(is)}{\partial ^2 \over \partial x_j\partial x_k},
\]
where ${\mathcal A}_{jk}^{(is)}$ is the element of the matrix ${\mathcal A}_{jk}$ situated
at the intersection of the $i$-th row and the $s$-th column.

We consider the scalar equations 
$$
{\mathfrak A} _{ss} (D_x)u_s=0\;\;\hbox{and}\;\;{\mathfrak A} _{ps} (D_x)u_s=0
\;\;\;\hbox{in}\;\;{\mathbb R}^n_{+}
$$ 
with the boundary condition $u_s=f_0$ on $\partial {\mathbb R}^n_{+}$, 
where $p$ is a fixed element of the set $\{1,\dots,m \}$ and $p\neq s$.

By the original assumption, the operator ${\mathfrak A} _{0} (D_x)$ is
strongly elliptic, so the operator ${\mathfrak A} _{ss} (D_x)$ is elliptic.

Without loss of generality it can be assumed that ${\mathcal A}_{nn}^{(ss)}>0$.
Setting 
\[
x _n=\sqrt{{\mathcal A}_{nn}^{(ss)}} y_n\;, 
\]
we perform a linear change of variables that takes the operator 
${\mathfrak A} _{ss} (D_x)$ to the canonical form
\begin{equation} \label{EELMSR_0.13F}
\tilde{{\mathfrak A}} _{ss} (D_y)=\sum ^n_{i=1}{\partial ^2 \over \partial y_i ^2}\;.
\end{equation}

Assume that the function $f_0$ in (\ref{EELMSR_0.11}) has compact support. If we
apply the Fourier transform\index{Fourier transform} with respect to the variables 
$y_1,\dots, y_{n-1}$ to the equation $\tilde{\mathfrak A} _{ss} (D_y)\tilde{u}_s(y)=0$
then we obtain
\[
{d^2 F[\tilde{u}_s] \over dy^2_n} -|\bs\xi '|^2 F[\tilde{u}_s] =0\;,
\]
where $\bs\xi '=(\xi _1 ,\dots,\xi _{n-1})$ and $|\bs\xi '|=(\xi _1 ^2+\dots+\xi _{n-1}^2)^{1/2}$.
The last equation implies
\[
F[\tilde{u}_s](\bs\xi ', y_n)=F[\tilde{u}_s](\bs\xi ', 0)\exp(-|\bs\xi '|y_n)=
F[\tilde{f}_0](\bs\xi ')\exp(-|\bs\xi '|y_n)\;.
\]

At the same time we transform the equation
${\mathfrak A} _{ps} (D_x)u_s =0$ to the variables $y_1,\dots,y_n$, and
then we apply to it the Fourier transform\index{Fourier transform} with 
respect to $y_1,\dots,y_{n-1}$. As a result,
\begin{equation} \label{EELMSR_0.12}
\tilde{{\mathcal A}}^{(ps)} _{nn}{d^2 F[\tilde{u}_s] \over dy^2_n} -2i{d F[\tilde{u}_s] \over dy_n}
\sum_{j=1}^{n-1} \tilde{{\mathcal A}}^{(ps)} _{jn}\xi _j -  F[\tilde{u}_s]
\sum_{j,k=1}^{n-1} \tilde{{\mathcal A}}^{(ps)} _{jk}\xi _j \xi _k=0.
\end{equation}
From the last equation and the equality $F[\tilde{u}_s](\bs\xi ', y_n)=F[\tilde{f}_0 ]
(\bs\xi ')\exp(-|\bs\xi '|y_n)$ we conclude that
\[
\sum_{s=1}^{n-1} \tilde{{\mathcal A}}^{(ps)} _{jn}\xi _j =0,
\]
i.e., $\tilde{{\mathcal A}}^{(ps)} _{jn}=0$ for all $j=1,\dots,n-1$. Therefore, 
differentiating $F[\tilde{u}_s](\bs\xi ', y_n)$ with respect to $y_n$ and 
substituting the result in (\ref{EELMSR_0.12}), 
we find that
\[
\tilde{{\mathcal A}}^{(ps)} _{nn}|\bs\xi '|^2 - \sum_{j,k=1}^{n-1} \tilde{{\mathcal A}}^{(ps)} _{jk}\xi _j \xi _k=0.
\]
Hence, $ \tilde{{\mathcal A}}^{(ps)} _{jk}=\delta _{jk} \tilde{{\mathcal A}}^{(ps)} _{nn},
\; 1 \leq  j,k \leq n-1$.
Thus, the operator $\tilde{{\mathfrak A}} _{ps} (D_ y)$ turns out to be the
Laplacian (up to a constant factor) , i.e.,
\begin{equation} \label{EELMSR_0.13}
\tilde{{\mathfrak A}}_{ps}(D_y)= \tilde{{\mathcal A}}^{(ps)} _{nn} 
\sum ^n_{i=1}{\partial ^2 \over \partial y_i ^2}\;.
\end{equation}
The inverse transformation of variables $y\rightarrow x$ in (\ref{EELMSR_0.13}) and
(\ref{EELMSR_0.13F}) gives
\begin{equation} \label{EELMSR_0.14}
{\mathfrak A}_{ps}(D_x)=b_{ps}{\mathfrak A}_{ss}(D_x)=
b_{ps}\sum ^n_{j,k=1}{\mathcal A}_{jk}^{(ss)}{\partial ^2 \over \partial x_j\partial x_k}
\end{equation}
for all $ p\neq s$. Taking into account the arbitrariness of $s \in \{1,\dots, m \}$,
we arrive at (\ref{diag1}), where $A=((b_{ps}))$ and 
$a^{(s)}_{jk}={\mathcal A}_{jk}^{(ss)}$. The non-singularity of the matrix $A$ follows
from the strong ellipticity of the operator ${\mathfrak A}_0(\partial /\partial x)$.

(ii) {\sl Sufficiency.} By Theorem \ref{T_2} together with Lemma \ref{1A}, and the 
equivalence of the systems
$$
A\;{\rm diag}\{L_1({D_x}),\dots, L_m({D_x}) \}\bs u=\bs 0
$$
and
$$
{\rm diag}\{L_1({D_x}),\dots, L_m({D_x}) \}\bs u=\bs 0\;,
$$
we conclude that representation (\ref{diag1}) is sufficient for the invariance 
of the orthant ${\mathbf R}^m_+(\alpha_1,\dots\alpha_m)$ for the system 
${\mathfrak A}_0({D_x})\bs u=\bs 0$ in ${\mathbb R}^n_+$.
\blankbox

The proof of Theorem \ref{T_4} is quite similar to the proof of Theorem \ref{T_3}
with some distinctions. Namely, the proof of necessity in Theorem \ref{T_4} starts,
by Proposition \ref{P_1} and Lemma \ref{4A}, with representation 
of a unique solution of the Dirichlet problem (\ref{EELMSR_0.02}) in the form
$$
\bs u(x)=\int _{\partial {\mathbb R}^n_{+}} {\Phi}\left ({y-x\over |y-x|}\right )
{x_n\over |y-x|^n}\bs f(y')dy'
$$
instead of representation (\ref{EELMSR_0.10}), where $\Phi $ is diagonal $(m\times m)$-matrix-valued kernel with the only
element $\varphi $ on the main diagonal. So, ${\mathfrak D}_1=\dots={\mathfrak D}_m=\varphi$ in (\ref{EELMSR_0.11})
and, therefore, $u_1=\dots=u_m$. Then, using similar arguments as in the proof 
of Theorem \ref{T_3}, we arrive at the equalities  
$$
{\mathfrak A}_{ps}(D_x)=b_{ps}L(D_x)
$$
instead of (\ref{EELMSR_0.14}), where $p, s \in \{1,\dots,m\}$ and $L(D_x)$ is a 
scalar elliptic operator. 




\begin{thebibliography}{99}

\bibitem  
{ADN2} S. Agmon, A. Douglis, and L. Nirenberg,
\textit{Estimates near the boundary for solutions of elliptic partial
differential equations satisfying general boundary conditions, II.},
Comm. Pure Appl. Math., \textbf{17} (1964), 35--92.

\bibitem 
{ALIK1} N. Alikakos, \textit{Remarks on invariance in reaction-diffusion equations}, 
Nonlinear Analysis. Theory, Methods \& Applications, 
\textbf{5}:6 (1981), 593--614.

\bibitem 
{AM} H. Amann, \textit{Invariant sets and existence theorems
for semilinear parabolic and elliptic systems}, J. Math. Anal. Appl.,
\textbf{65} (1978), 432--467.

\bibitem 
{Bates} P.W. Bates, \textit{Containment for weakly coupled parabolic systems}, 
Houston J. Math., \textbf{11}:2 (1985), 151--158.

\bibitem 
{BCF} J.W. Bebernes, K.N. Chueh, and W. Fulks, \textit{Some applications of invariance 
for parabolic systems}, Indiana Univ. Math. J., \textbf{28}:2 (1979), 269--277.

\bibitem 
{BS} J.W. Bebernes and K. Schmitt, \textit{Invariant sets and the
Hukuhara-Kneser property for systems of parabolic partial differential
equations}, Rocky Mountain J. Math., \textbf{7} (1977), 557--567.

\bibitem
{BM} Yu.D. Burago and V.G. Maz'ya, \textit{Potential Theory and 
the Function Theory for Irregular Regions}, Zap. Nauchn. Sem. LOMI, \textbf{3} (1967); 
English translation: Sem. in Mathematics,  V.A. Steklov Math. Inst., Leningrad, 
\textbf{3}, Consultants Bureau, New York, 1969.

\bibitem 
{CCS} K.N. Chueh, C.C. Conley, and  J.A. Smoller,
\textit{Positively invariant regions for systems of nonlinear diffusion equations},
Indiana Univ. Math. J., \textbf{26} (1977), 373--391.

\bibitem 
{CHS} E. Conway, D. Hoff, and J. Smoller, \textit{Large time behavior of solutions 
of systems of nonlinear reaction-diffusion equations}, SIAM J. Appl. Math.,
\textbf{35} (1) (1978), 1--16.

\bibitem 
{CS} C. Cosner and P.W. Schaefer, \textit{On the development
of functionals which satisfy a maximum principle}, Appl. Analysis,
\textbf{26} (1987), 45--60.

\bibitem
{GiTr} D. Gilbarg and N.S. Trudinger, \textit{Elliptic Partial 
Differential Equations of Second Order}, Springer Verlag, 
Berlin-Heidelberg-New York-Tokyo, 1983.

\bibitem 
{KM1} G.I. Kresin and V.G. Maz'ya,  \textit{Criteria for
validity of the maximum modulus principle for solutions of linear
parabolic systems}, Ark. Math., {\bf 32} (1994), 121--155.

\bibitem 
{KM2} G.I. Kresin and V.G. Maz'ya, \textit{On the maximum
principle with respect to smooth norms for linear strongly coupled
parabolic systems}, Functional Differential Equations, {\bf 5}:3-4 (1998),
349--376.

\bibitem 
{KM3} G.I. Kresin and V.G. Maz'ya, \textit{Criteria for
validity of the maximum norm principle for parabolic systems},
Poten. Anal., {\bf 10} (1999), 243--272.

\bibitem 
{KM} G. Kresin and V. Maz'ya, \textit{Maximum Principles and Sharp 
Constants for Solutions of Elliptic and Parabolic Systems},
Math. Surveys and Monographs, \textbf{183}, Amer. Math. Soc., Providence, 
Rhode Island, 2012.

\bibitem 
{KM4} G. Kresin and V. Maz'ya, \textit{Criteria for invariance of convex 
sets for linear parabolic systems}, Proceedings of the International Conference 
"Complex Analysis and Dynamical Systems VI", Contemporary Mathematics, 
Amer. Math. Soc. (to be published).

\bibitem 
{Kuip} H.J. Kuiper, \textit{Invariant sets for nonlinear elliptic and parabolic systems},  
SIAM J. Math. Anal., {\bf 11}:6 (1980), 1075--1103.

\bibitem 
{LAN} E.M. Landis, \textit{Second Order Equations of Elliptic and Parabolic Type},
Transl. of Math.  Monographs, \textbf{171}, Amer. Math. Soc., Providence, 
Rhode Island, 1998.

\bibitem 
{Lem} R. Lemmert, \textit{\"Uber die Invarianz konvexer Teilmengen eines normierten
Raumes in bezug auf elliptische Differentialgleichungen},
Comm. Partial Diff. Eq., {\bf 3}:4 (1978), 297--318.

\bibitem 
{LO1} Ya.B. Lopatinski\v{\i}, \textit{On a method of reducing boundary value
problems for systems of differential equations of elliptic type to regular integral
equations}, Ukrain. Mat. \v{Z}urnal, {\bf 5}:2 (1953), 123--151 (Russian).
\index{Lopatinski\v{\i}, Ya.B.}

\bibitem 
{MK} V.G. Maz'ya and G.I. Kresin, \textit{On the maximum
principle for strongly elliptic and parabolic second order systems
with constant coefficients}, Mat. Sb., {\bf 125(167)} (1984), 458--480 (Russian);
English transl.: Math. USSR Sb., {\bf 53} (1986) ,457--479.

\bibitem 
{Ots} K. Otsuka, \textit{On the positivity of the fundamental solutions for parabolic 
systems}, J. Math. Kioto Univ., {\bf 28} (1988), 119--132.

\bibitem 
{PW} M.H. Protter and H.F. Weinberger, \textit{Maximum
Principles in Differential Equations}, 
Prentice-Hall, Inc., Englewood Cliffs, N.J., 1967;
Springer-Verlag, New York Inc., 1984.

\bibitem 
{RW1} R. Redheffer and W. Walter, \textit{Invariant sets
for systems of partial differential equations. I. Parabolic equations},
Arch. Rat. Mech. Anal., {\bf 67} (1978), 41--52.

\bibitem 
{RW2} R. Redheffer and W. Walter, \textit{Invariant sets
for systems of partial differential equations. II. First-order and elliptic equations},
Arch. Rat. Mech. Anal., {\bf 73} (1980), 19--29.

\bibitem 
{ROCK} R.T. Rockafellar, \textit{Convex Analysis}, Princeton Univ.
Press, Princton, N.J., 1970.

\bibitem 
{Shaef} C. Schaefer, \textit{Invariant sets and contractions for weakly coupled
systems of parabolic differential equations}, Rend. Mat.,
\textbf{13} (1980), 337--357.

\bibitem 
{SHA} Z.Ya. Shapiro, \textit{The first boundary value problem for an
elliptic system of differential equations}, Mat. Sb., {\bf 28(70)}:1 (1951),
55-78 (Russian).

\bibitem 
{Smoller} J. Smoller, \textit{Shock Waves and Reaction-Diffusion Equations},
Springer-Verlag, Berlin-Heidelberg-New York, 1983.

\bibitem 
{SOL} V.A. Solonnikov, \textit{On general boundary value problems
for systems elliptic in the Douglis-Nirenberg sense}, Izv. Akad. Nauk SSSR, ser.
Mat., {\bf 28}:3 (1964), 665--706 ; English transl. Amer. Math. Soc. Transl. (2), 
{\bf 56} (1966), 193--232.

\bibitem 
{WAL} W. Walter, \textit{Differential and Integral Inequalities},
Springer, Berlin-Heidelberg-New York, 1970.

\bibitem 
{WEIN} H.F. Weinberger, \textit{Invariant sets for weakly coupled
parabolic and elliptic systems}, Rend. Mat., {\bf 8} (1975), 295--310.

\bibitem 
{WEIN1} H.F. Weinberger, \textit{Some remarks on invariant sets for systems}, 
in ``Maximum Principles and Eigenvalue Problems in Partial Differential Equations'', 
P.W. Schaefer ed., Pittman Research Notes, Math. ser., {\bf 175} (1988), pp. 189--207.
\end{thebibliography}
\end{document}